\documentclass[10pt,reqno]{amsart}
\usepackage{pdfsync}
\usepackage{amssymb,amscd}
\usepackage{color}

\usepackage[bookmarks=true,hyperindex,pdftex,colorlinks,citecolor=blue]{hyperref}

\hypersetup{pdftitle={Norm-attaining compact operators},
pdfauthor={Miguel Martin}}

\theoremstyle{plain}
\newtheorem{theorem}{Theorem}%[section]
\newtheorem{proposition}[theorem]{Proposition}
\newtheorem{prop}[theorem]{Proposition}
\newtheorem{corollary}[theorem]{Corollary}
\newtheorem{lemma}[theorem]{Lemma}

\theoremstyle{definition}

\newtheorem{question}[theorem]{Question}

 \newcommand{\N}{\mathbb{N}}
  \newcommand{\C}{\mathbb{C}}

 \renewcommand{\geq}{\geqslant}
 \renewcommand{\leq}{\leqslant}

\begin{document}
\title{Norm-attaining compact operators}

\author[Miguel Martin]{}

  \thispagestyle{plain}

\thanks{Supported by Spanish MICINN and FEDER project no.\ MTM2012-31755 and by Junta de Andaluc\'{\i}a and FEDER grants FQM-185 and P09-FQM-4911.}


\date{June 5th, 2013. Revised May 16th, 2014}

\maketitle

\centerline{\textsc{\large Miguel Mart\'{\i}n}}

\begin{center}\small Departamento de An\'{a}lisis Matem\'{a}tico \\ Facultad de
Ciencias \\ Universidad de Granada \\ 18071 Granada, Spain \\
\emph{E-mail:} \texttt{mmartins@ugr.es}
\end{center}

\begin{abstract}
We show examples of compact linear operators between Banach spaces which cannot be approximated by norm attaining operators. This is the negative answer to an open question posed in the 1970's. Actually, any strictly convex Banach space failing the approximation property serves as the range space. On the other hand, there are examples in which the domain space has a Schauder basis.
\end{abstract}

\section{Introduction}

Motivated by the classical Bishop-Phelps theorem of 1961 \cite{BishopPhelps} stating the density of norm-attaining functionals on every Banach space, the study of the density of norm-attaining operators started with J.~Lindenstrauss' 1963 paper \cite{Lindens}, where the author showed that the Bishop-Phelps theorem is not longer true for operators and gave some partial positive results. We recall that an operator $T$ between two Banach spaces $X$ and $Y$ is said to \emph{attain its norm} whenever there is $x\in X$ with $\|x\|=1$ such that $\|T\|=\|T(x)\|$ (i.e.\ the supremum defining the operator norm is actually a maximum). An intensive research about this topic has been developed by, among others, J.~Bourgain in the 1970's, J.~Partington and W.~Schachermayer in the 1980's, and M.~Acosta, W.~Gowers and R.~Pay\'{a} in the 1990's. The expository paper \cite{Acosta-RACSAM} can be used for reference and background.

All known examples of operators which cannot be approximated by norm-attaining ones are non-compact, so the question whether every linear compact operator between Banach spaces can be approximated by norm-attaining operators seems to be open. It was explicitly asked by J.~Diestel and J.~Uhl in the 1976 paper \cite{Diestel-Uhl-Rocky} (as Problem~4 in page~6) and in their monograph on vector measures \cite[p.~217]{D-U}, and also in the 1979 paper by J.~Johnson and J.~Wolfe \cite{JoWo} (as Question~2 in page~17). More recently, the question also appeared in the 2006 expository paper by M.~Acosta \cite[p.~16]{Acosta-RACSAM}.

The main aim of this paper is to answer the question in the negative by providing two Banach spaces $X$ and $Y$ and a compact linear operator from $X$ into $Y$ which cannot be approximated by norm-attaining operators. This comes from extending an idea of Lindenstrauss for $c_0$ to its closed subspaces and applying it to Enflo's counterexample to the approximation problem. Moreover, thanks to an example of W.~Johnson and G.~Schechtman, the space $X$ can be taken with Schauder basis. It is also possible to get an example where $X=Y$. We also show that for every strictly convex Banach space $Y$ without the approximation property, there exists a Banach space $X$ such that $K(X,Y)$ is not contained in the closure of the set of norm-attaining operators. Finally, we present some known conditions ensuring the density of norm-attaining operators in the space of compact operators and discuss some open questions.

\vspace*{0.2cm}

Let us finish the introduction with the needed notation. Given two (real or complex) Banach spaces $X$ and $Y$, we write $L(X,Y)$ for the Banach space of all bounded linear operators from $X$ into $Y$, endowed with the operator norm. By $K(X,Y)$ and $F(X,Y)$ we denote the subspaces of $L(X,Y)$ of compact operators and finite-rank operators, respectively. We write $B_X$ to denote the closed unit ball of $X$. The set of all norm-attaining operators from $X$ into $Y$ is denoted by $NA(X,Y)$.

\section{The results}\label{sec:counterexamples}

Let us start with the promised counterexample.

\begin{theorem} \label{fact}
There exist compact linear operators between Banach spaces which cannot be approximated by norm-attaining operators.
\end{theorem}

The idea for the proof of the above result comes from extending (the proof of) \cite[Proposition~4]{Lindens} to closed subspaces of $c_0$ and then apply it to Enflo's counterexample to the approximation problem. We state the first ingredient for further use. Recall that a Banach space $Y$ is said to be \emph{strictly convex} if the unit sphere of $Y$ fails to contain non-trivial segments, equivalently, if for every $y\in Y$ with $\|y\|=1$ and $z\in Y$, $\|y \pm z\|\leq 1$ implies $z=0$.

\begin{lemma}\label{lemma-subspace-c0-into-sc}
Let $X$ be a closed subspace of $c_0$ and let $Y$ be a strictly convex Banach space. Then, $NA(X,Y)\subseteq F(X,Y)$.
\end{lemma}

\begin{proof}
Fix $T\in NA(X,Y)$ and $x_0\in B_X$ such that $\|T(x_0)\|=\|T\|=1$. As $x_0\in c_0$, there is $N\in \N$ such that $|x_0(n)|<1/2$ for every $n\geq N$. Now, consider the subspace $Z$ of $X$ given by
$$
Z:=\bigl\{x\in X\,:\, x(i)=0\ \text{for } 1\leq i \leq N \bigr\}
$$
and observe that for every $z\in Z$ with $\|z\|\leq 1/2$, we have
$$
\left\|x_0 \pm z\right\|\leq 1.
$$
Therefore,
$$
\left\|T (x_0) \pm T(z)\right\|\leq 1
$$
and, being $Y$ strictly convex and $\|T(x_0)\|=1$, it follows that $T(z)=0$. Therefore, $T$ vanished on a finite-codimensional space.
\end{proof}

Prior to give the proof of the theorem, we have to recall the concept of (Grothendieck) approximation property. We refer to \cite{Linden-Tz} for background. A Banach space $X$ has the \emph{approximation property} if for every compact set $K$ and every $\varepsilon>0$, there is $R\in F(X,X)$ such that $\|x-R(x)\| <\varepsilon$ for all $x\in K$. It was shown by P.~Enflo in 1973 that there are Banach spaces failing the approximation property showing, actually, that there are closed subspaces of $c_0$ without the approximation property.

\begin{proof}[Proof of Theorem~\ref{fact}] Let $X$ be a closed subspace of $c_0$ failing the approximation property (Enflo's example works, see \cite[Theorem~2.d.6]{Linden-Tz}). Then, $X^*$ also fails the approximation property so there is a Banach space $Y$ and a compact operator $T:X\longrightarrow Y$ which cannot be approximated by finite-rank operators (see \cite[Theorem~1.e.5]{Linden-Tz}). As we may clearly suppose that $Y$ is separable (considering the closure of $T(X)$) and the approximation property is of isomorphic nature, we may and do suppose that $Y$ is strictly convex (recall that every separable Banach space admits a strictly convex equivalent renorming by an old result of V.~Klee, see \cite[\S II.2]{DGZ}). Now, Lemma~\ref{lemma-subspace-c0-into-sc} shows that $T$ cannot be approximated by norm-attaining operators.
\end{proof}

Next, we would like to present two ways to obtain examples as in Theorem~\ref{fact}. First, with respect to domain spaces, we observe that the above proof works for arbitrary closed subspaces of $c_0$ whose dual fails the approximation property.

\begin{prop}
For every closed subspace $X$ of $c_0$ such that $X^*$ fails the approximation property, there exist a Banach space $Y$ and a compact linear operator from $X$ into $Y$ which cannot be approximated by norm-attaining operators.
\end{prop}

Using the result due to W.~Johnson and G.~Schechtman \cite[Corollary~JS, p.~127]{Joh-Oik} that there is a closed subspace of $c_0$ with Schauder basis whose dual fails the approximation property, we may state the following corollary.

\begin{corollary}\label{coro:Schauder-JS}
There exist a Banach space $X$ with Schauder basis, a Banach space $Y$ and a compact linear operator $T$ between $X$ and $Y$ which cannot be approximated by norm-attaining operators.
\end{corollary}

Dealing with range spaces, the idea of Theorem~\ref{fact} can be also squeezed to show that for every strictly convex Banach space $Y$ without the approximation property, an example of the same kind can be constructed. We will use the following characterization of the approximation property, known to A.~Grothendieck (see ``Proposition'' 37 in p.~170 of \cite{Gro-MAMS}), which follows easily from the compact factorization of every compact operator through a closed subspace of $c_0$. A proof of the lemma can be found in \cite[Theorem~18.3.2]{Jarchow}.

\begin{lemma}[Grothendieck]\label{lemma-Grothendieck} A Banach space $Y$ has the approximation property if and only if $F(X,Y)$ is dense in $K(X,Y)$ for every closed subspace $X$ of $c_0$.
\end{lemma}

We are now able to present the promised result.

\begin{prop}\label{theo:every-strictly-convex}
Let $Y$ be a strictly convex Banach space without the approximation property. Then, there exist a Banach space $X$ and a compact linear operator from $X$ into $Y$ which cannot be approximated by norm-attaining operators.
\end{prop}

\begin{proof}
By Lemma~\ref{lemma-Grothendieck}, there is a closed subspace $X$ of $c_0$ such that $F(X,Y)$ is not dense in $K(X,Y)$. But Lemma~\ref{lemma-subspace-c0-into-sc} implies that $NA(X,Y)\subset F(X,Y)$, so there are compact operators from $X$ into $Y$ which cannot be approximated by norm-attaining operators.
\end{proof}

Compare the result above with the one by M.~Acosta \cite{Aco-Edinburgh} of 1999, stating that there is a Banach space $X$ such that for every infinite-dimensional strictly convex Banach space $Y$, there exists a (non-compact) operator $T\in L(X,Y)$ which cannot be approximated by norm-attaining operators. The case when $Y=\ell_p$ was previously done by W.~Gowers \cite{Gowers} in 1990.

Next, we would like to give a result for subspaces of complex $L_1(\mu)$ spaces. We first need to recall the notion of complex strict convexity. A complex Banach space $Y$ is said to be \emph{complex strictly convex} if for every $y\in Y$ with $\|y\|=1$ and $z\in Y$, the condition $\|y + \theta z\|\leq 1$ for every $\theta\in \C$ with $|\theta|=1$ implies $z=0$. Clearly, strictly convex spaces are complex strictly convex, but the converse is false, as $L_1(\mu)$ spaces are complex strictly convex, see \cite[Proposition~3.2.3]{Istra}. By an obvious adaption of the proof of Lemma~\ref{lemma-subspace-c0-into-sc}, we get that, in the complex case, {\slshape if $X$ is a closed subspace of $c_0$ and $Y$ is a complex strictly convex space, then $NA(X,Y)\subseteq F(X,Y)$.}\  Therefore, the following result follows with the same proof than Theorem~\ref{theo:every-strictly-convex}.

\begin{proposition}
Let $\mu$ be a measure and let $Y$ be a closed subspace of the complex space $L_1(\mu)$ without the approximation property. Then, there exist a Banach space $X$ and a compact linear operator from $X$ into $Y$ which cannot be approximated by norm-attaining operators.
\end{proposition}

We do not know whether this result is also true in the real case. It is known that there is a Banach space $X$ such that for every measure $\mu$ such that $L_1(\mu)$ is infinite dimensional, there is a (non-compact) operator $T$ from $X$ into $L_1(\mu)$ which cannot be approximated by norm-attaining operators (M.~Acosta, \cite{Aco-contemporary}).

Our next result provides with an example in which the domain and the range space coincides.

\begin{theorem}
There exist a Banach space $Z$ and a compact operator from $Z$ into $Z$ which cannot be approximated by norm-attaining operators.
\end{theorem}

\begin{proof}
Let $X$ and $Y$ be Banach spaces and fix $T_0\in K(X,Y)$ with $\|T_0\|=1$ and $0<\varepsilon <1/2$. Write $Z=X\oplus_\infty Y$ (i.e.\ $\|(x,y)\|=\max\{\|x\|,\|y\|\}$ for $(x,y)\in X\times Y$) and define $S_0\in K(Z,Z)$ by $S_0(x,y)=(0,T_0(x))$ for every $(x,y)\in X\oplus_\infty Y$, which clearly satisfies $\|S_0\|=1$. We claim that if there is an operator $S\in NA(Z,Z)$ such that $\|S_0-S\|<\varepsilon$, then there is $T\in NA(X,Y)$ such that $\|T_0-T\|<\varepsilon$. Indeed, take $(x_0,y_0)\in B_Z=B_X\times B_Y$ such that $\|S(x_0,y_0)\|=\|S\|$ and write $P_1:Z\longrightarrow X$ and $P_2:Z\longrightarrow Y$ for the natural projections. Now, observe that
$$
\|P_1 S\|=\|P_1 S - P_1 S_0\|\leq \|S-S_0\|<\varepsilon <1/2
$$
so, as $\|S\|\geq 1 -\varepsilon >1/2$, we get that
$$
\|P_2 S(x_0,y_0)\|=\|P_2 S\|=\|S\|.
$$
Next, take $x_0^*\in S_{X^*}$ such that $x_0^*(x_0)=1$ and define the operator $T\in L(X,Y)$ by
$$
T(x)=P_2S\bigl(x,x_0^*(x)y_0\bigr) \qquad \bigl(x\in X\bigr).
$$
Then, $\|T\|\leq \|P_2 S\|$ and $\|T(x_0)\|=\|P_2 S(x_0,y_0)\|= \|P_2 S\|$, so $T\in NA(X,Y)$. On the other hand, for $x\in B_X$,
\begin{align*}
\|T_0(x)-T(x)\|&=\bigl\|P_2 S_0(x,x_0^*(x)y_0)\,- \,P_2S(x,x_0^*(x)y_0)\bigr\|\\ &\leq \|P_2S_0-P_2S\|\leq \|S_0-S\|<\varepsilon,
\end{align*}
as claimed.

Now, if we take $X$, $Y$, and $T_0\in K(X,Y)$ which cannot be approximated by norm-attaining operators, then $Z=X\oplus_\infty Y$ and $S_0\in K(Z,Z)$ defined as above, give the desired example.
\end{proof}

Let us finish the paper with an small discussion about positive results on norm-attaining compact operators. The main open question here (and also in the general theory of norm-attaining operators) is whether finite-rank operators can be always approximated by norm-attaining operators.

\begin{question}
Let $Y$ be a finite-dimensional space. Is it true that for every Banach space $X$, $NA(X,Y)$ is dense in $L(X,Y)$?
\end{question}

We would like to comment that the problem above is open even when $Y$ is the $2$-dimensional real Hilbert space. Related to this, let us also comment that there is a complex version of the Bishop-Phelps theorem which states that for every complex Banach space $X$, complex-linear norm-attaining operators (i.e.\ functionals) from $X$ into $\C$ are dense in the space of all complex-linear operators (see \cite{Phelps1} or \cite[\S 2]{Phelps2}). On the other hand, V.~Lomonosov showed in 2000 \cite{Lomonosov} that there is a complex Banach space $X$ and a (non-complex symmetric) closed convex bounded subset $C$ of $X$ such that there is no element in $X^*$ attaining the supremum of its modulus on $C$.

We now present some conditions on the domain space assuring the density of norm-attaining compact operators. First, we recall the celebrated paper by J.~Bourgain \cite{Bourgain} about dentability in which it is proved that given a Banach space $X$ with the Radon-Nikod\'{y}m property and a Banach space $Y$, every operator from $X$ into $Y$ can be
approximated by compact perturbations of it attaining the norm. Therefore, the following result clearly follows.

\begin{proposition}[Bourgain]
Let $X$ be a Banach space with the Radon-Nikod\'{y}m property. Then for every Banach space $Y$, $NA(X,Y)\cap K(X,Y)$ is dense in $K(X,Y)$.
\end{proposition}

For spaces failing the Radon-Nikod\'{y}m property, J.~Diestel and J.~Uhl (1976) \cite{Diestel-Uhl-Rocky} showed that norm-attaining finite-rank operators from $L_1(\mu)$ into any Banach space are dense in the space of all compact operators. This study was continued by J.~Johnson and J.~Wolfe \cite{JoWo} (1979), who proved, among other things, the same result for real $C(K)$ spaces. This last result is proved using a stronger version of the approximation property of the dual. We include the proof of the following result (which is omitted in \cite{JoWo}) for completeness.

\begin{proposition}[\textrm{\cite{JoWo}}]\label{prop-suficiente}
Let $X$ be a Banach space. Suppose there is a net $(P_\alpha)$ of finite-rank contractive projections on $X$ such that for every $x^*\in X^*$, $(P_\alpha^* x^*)\longrightarrow x^*$ in norm. Then for every Banach space $Y$, $NA(X,Y)\cap K(X,Y)$ is dense in $K(X,Y)$.
\end{proposition}

\begin{proof}
Let $Y$ be a Banach space and consider $T\in K(X,Y)$. For every $\alpha$, the operator $T P_\alpha$ attains its norm since $TP_\alpha(B_X)=T(B_{P_\alpha(X)})$ (here we use that $P_\alpha$ is a norm-one projection) and $B_{P_\alpha(X)}$ is compact. We claim that $(TP_\alpha)\longrightarrow T$ in the operator norm, finishing the proof. Indeed, given $\varepsilon>0$, as $T^*$ is compact, we may find an $\varepsilon/3$-net $x^*_1,\ldots,x_n^*\in X^*$ for $T^*(B_{Y^*})$ and we may find $\alpha_0$ such that  $\|P_\alpha^*(x_i^*)-x_i^*\|<\varepsilon/3$ for $i=1,\ldots,n$ and every $\alpha\geq \alpha_0$. Now, given $y^*\in B_{Y^*}$, we take $i\in \{1,\ldots,n\}$ such that $\|T^*(y^*)-x_i^*\|<\varepsilon/3$ and observe that
$$
\|P_\alpha^*T^*(y^*)-T^*(y^*)\|\leq \|P_\alpha^*T^*(y^*)-P_\alpha^*(x_i^*)\| + \|P_\alpha^*(x_i^*)-x_i^*\| + \|x_i^*-T^*(y^*)\|<\varepsilon.
$$
In other words, $\|TP_\alpha-T\|=\|P_\alpha^*T^*-T^*\|\leq \varepsilon$ for every $\alpha\geq \alpha_0$.
\end{proof}

It is shown in \cite[Proposition~3.2]{JoWo} that every $C(K)$ space satisfies the condition of the above proposition. The following is also a consequence of the proposition.

\begin{corollary}\label{coro-subc0monotone}
Let $X$ be a closed subspace of $c_0$ with a monotone Schauder basis. Then, for every Banach space $Y$, $NA(X,Y)\cap K(X,Y)$ is dense in $K(X,Y)$.
\end{corollary}

\begin{proof}
As $c_0$ is an $M$-embedded space \cite[Examples~III.1.4]{HWW} and $M$-embeddedness passes to closed subspaces \cite[Theorem~III.1.6]{HWW}, we may use the a result of G.~Godefroy and P.~Saphar that Schauder bases in $M$-embedded spaces with basis constant less than $2$ are shrinking (see \cite[Corollary~III.3.10]{HWW}, for instant). Then, $X$ possesses a shrinking Schauder basis and so Proposition~\ref{prop-suficiente} applies.
\end{proof}

Compare this result with the example given in Corollary~\ref{coro:Schauder-JS} of a closed subspace $X$ of $c_0$ with Schauder basis such that there is a compact operator defined on $X$ which cannot be approximated by norm-attaining operators. On the other hand, as far as we know, the following question remains open.

\begin{question}
Let $X$ be a closed subspace of $c_0$ with the metric approximation property. Is it true that for every Banach space $Y$, $NA(X,Y)$ is dense in $L(X,Y)$?
\end{question}

\vspace*{0.5cm}

\noindent \textbf{Acknowledgment:\ } The author is grateful to Rafael Pay\'{a} for many fruitful conversations about the content and the form of this paper. He also thanks Joe Diestel, Ioannis Gasparis, Bill Johnson, Gilles Godefroy, and Gin\'{e}s L\'{o}pez for kindly answering several inquiries. The author also thanks the anonymous referee for helpful suggestions about revision.

Research supported by Spanish MICINN and FEDER project no.\ MTM2012-31755 and by Junta de Andaluc\'{\i}a and FEDER grants FQM-185 and P09-FQM-4911.

\end{document}